\documentclass[12pt]{amsart}
\usepackage{amsmath,amsthm,latexsym,amscd,amsbsy,amssymb}
 \relax

\chardef\bslash=`\\ 

\makeatletter
\def\verbatim{\interlinepenalty\@M \@verbatim
  \leftskip\@totalleftmargin\advance\leftskip2pc
  \frenchspacing\@vobeyspaces \@xverbatim}
\makeatother
\newtheorem{thm}{Theorem}[section]
\newtheorem{cor}[thm]{Corollary}

\newtheorem{pro}[thm]{Proposition}

\newtheorem{rem}[thm]{Remark}

\newtheorem{que}[thm]{Question}

\begin{document}


\title
{Function Spaces over Products with Ordinals}
\author{Raushan  Buzyakova}
\email{Raushan\_Buzyakova@yahoo.com}



\begin{abstract}{
We are concerned with the question of when a homeomorphism between $C_p(X\times \tau)$ and $C_p(Y\times \tau)$ implies the existence of a homeomorphism between $C_p(X\times \tau')$ and $C_p(Y\times \tau')$, where $\tau$ and $\tau'$ are spaces of ordinals with some additional desired properties.
}
\end{abstract}

\maketitle
\markboth{R. Buzyakova}{Function Spaces over Products with Ordinals}
{ }

\section{Introduction}\label{S:introduction}
\par\bigskip\noindent
We will be concerned with the following general problem.

\par\bigskip\noindent
{\bf Problem.} {\it
Let $C_p(X\times Z)$ be homeomorphic to $C_p(Y\times Z)$. Suppose that $Z'$ is a subspace or a superspace of $Z$. What additional conditions on $X,Y,Z,Z'$ guarantee that $C_p(X\times Z')$ is homeomorphic to $C_p(Y\times Z')$? That $X$ is homeomorphic to $Y$?
}

\par\bigskip\noindent
Note that for any $X$ and $Y$ we can find $Z$ such that $X\times Z$ is homeomorphic to $Y\times Z$ (in particular, $Z=X^\omega\times Y^\omega$ works). Therefore,  it is natural to require that  candidates for $Z$ in the above problems have rigid structures. There are many interesting results around this problem. Since most of $Z$'s  in this study will be ordered, we would like to mention O. Okunev's result (see, for example, \cite[the argument of Theorem 0.6.2]{Arh} ) that if $C_p(X\times \mathbb R)$ and $C_p(Y\times \mathbb R)$ are homeomorphic, then so are $C_p(X)^\omega$ and $C_p(Y)^\omega$.
In this work, we will primarily target  cases when $Z$ is a space of ordinals. Our major restriction on the topology of a space will be a network size. We will show, in particular, that under certain restrictions on a network size of $X$ and $Y$, a homeomorphism between $C_p(X\times \tau)$ and $C_p(Y\times \tau )$, where $\tau$ is regular and uncountable, implies the existence of a homeomorphism between $C_p(X\times \lambda)$ and $C_p(Y\times \lambda)$ for some $\lambda<\tau$.  This statement and other observations lead to some natural generalizations as well as  questions for possible further exploration.
\par\smallskip\noindent 
Recall that  a network of $X$  is a family of sets with the same properties as a basis except that elements of the family need not be open.  By $A(\tau)$ we denote the Alexandroff single-point compactification of a discrete space of size $\tau$ with $\infty$ being its only non-isolated point. All spaces under consideration are Tychonov. In notation and terminology of general topological nature we follow \cite{Eng}. To distinguish ordered pairs from open intervals, we will denote the former by $\langle a,b\rangle$ and the latter by $(a,b)$. We also refer the reader to \cite{Arh} for general facts and notations involving function spaces endowed with the topology of point-wise convergence. 

\section{Study}\label{S:study}

\par\bigskip\noindent
In our arguments we will use  the following folklore fact for which we outline a proof for convenience.

\par\bigskip\noindent
\begin{pro}\label{pro:folklore}
(Folklore) 
 Let an ordinal $\tau$ have cofinality greater than the density of an infinite space $X$ and let $f\in C_p(X\times \tau)\cup C_p(X\times (\tau +1))$.  Then there exists $\lambda <\tau$ such that $f|_{\{x\}\times [\lambda, \tau)}$ is constant for every $x\in X$.
\end{pro}
\begin{proof}
To see why the statement is true let $D=\{d_\alpha:\alpha <d(X)\}$ be dense in $X$, where $d(X)$ is the density of $X$. For each $\alpha<d(X)$, fix $\tau_\alpha<\tau$ such that $f$ is constant on $\{d_\alpha\}\times [\tau_\alpha, \tau)$. Such a $\tau_\alpha$ exists due to countable compactness of $\tau$. Since $d(X) < cf(\tau )$, we conclude that  $\lambda=\sup\{\tau_\alpha:\alpha<d(X)\}<\tau$. By continuity of $f$ and density of $D$, $\lambda$ is as desired.
\end{proof}

\par\bigskip\noindent
\begin{thm}\label{thm:timesomegaone}
Let $X$ and $Y$ have networks of cardinality less than $\tau$, where $\tau$ is regular and uncountable. Let $C_p(X\times \tau)$ be homeomorphic to $C_p(Y\times \tau)$. Then $C_p(X\times \lambda)$ is homeomorphic to $C_p(Y\times \lambda)$ for some non-zero $\lambda<\tau$.
\end{thm}
\begin{proof}
Fix a homeomorphism $H:C_p(X\times\tau)\to C_p(Y\times \tau)$. Next for $\alpha<\tau$, define $F(X,\alpha )$ as follows: $f\in F(X,\alpha )$ if and only if $f$ is constant on $\{x\}\times [\alpha,\tau)$ for every $x\in X$. Similarly, we define $F(Y,\alpha )$.

\par\medskip\noindent
{\it Claim 1. $F(X,\mu )$ is homeomorphic to $C_p(X\times (\mu + 1))$.}
\par\smallskip\noindent
To prove the claim, for each $f\in C_p(X\times (\mu + 1))$, let $f_\mu\in F(X,\mu)$ be the function such that $f(x,\alpha )=f_\mu(x,\alpha )$ for each $\langle x,\alpha\rangle\in X\times (\mu +1)$. Clearly, the map  $M: C_p(X\times (\mu + 1))\to 
 F(X,\mu)$ defined by $M(f)=f_\mu$ is a homeomorphism.

\par\bigskip\noindent
{\it Claim 2. $F(X,\alpha )$ is closed in $C_p(X\times \tau )$ for any $\alpha<\tau$}
\par\medskip\noindent
To prove the claim fix $f\in C_p(X\times \tau)\setminus F(X,\alpha)$. By the description of $F(X,\alpha)$, there exist  $x\in X$ and $\lambda\in (\alpha, \tau )$ such that $f(x,\lambda)\not = f(x,\alpha )$. Put $\epsilon=|f(x,\lambda)-f(x,\alpha )|$. Then the following open set contains $f$ and misses $F(X,\alpha)$:
$$
\{
g: g(x,\lambda )\in (f(x,\lambda) -\epsilon/3, f(x,\lambda) +\epsilon/3))\}\cap
$$
$$
\cap\{g: g(x,\alpha )\in (f(x,\alpha) -\epsilon/3, f(x,\alpha) +\epsilon/3))
\}
$$
The claim is proved.

\par\bigskip\noindent
{\it Claim 3. Let $\alpha<\tau$. Then $H(F(X,\alpha ))\subset F(Y,\alpha')$ for some $\alpha'<\tau$.}
\par\smallskip\noindent
To prove the claim, recall that  $X$ has netweight less than $\tau$ and $\tau$ is regular. Therefore, $X\times (\alpha +1)$ has netweight less than $\tau$. Therefore, $C_p(X\times (\alpha +1))$ has density less than $\tau$. By Claim 1, $F(X,\alpha )$ has density less than $\tau$ too. By Proposition \ref{pro:folklore},  any  subset of $C_p(Y\times \tau)$ of density less than $\tau$ is a subset of $F(Y,\alpha'))$ for some $\alpha'<\tau$.

\par\bigskip\noindent
By Claim 3, we can create a sequence $\alpha_1<\alpha_2<...$ with the following two properties:
\begin{description}
	\item[\rm P1] $H( F(X,\alpha_n))\subset F(Y,\alpha_{n + 1})$ if $n$ is odd.
	\item[\rm P2] $H^{-1}( F(Y,\alpha_n))\subset F(X,\alpha_{n + 1})$ if $n$ is even.
\end{description}
Put $\mu = \sup\{\alpha_n\}_n$. By Claim 1, it remains to show that  $H(F(X,\mu)) = F(Y,\mu )$. 
Since $\{\alpha_n\}_n$ is an increasing sequence we conclude that  that $H (\cup_n F(\alpha_n, X))= \cup_n F(\alpha_n,Y)$. Since $H $ is a homeomorphism, 
$H (\overline{\cup_n F(\alpha_n, X)})=\overline{\cup_n F(\alpha_n, Y)}$. It remains to show that $\overline{\cup_n F(\alpha_n, X)}=F(\mu,X)$
and $\overline{\cup_n F(\alpha_n, Y)}=F(\mu,Y)$. We will prove the former equality.
Since $F(\alpha_n,X)$ is a subset of $F(\mu,X)$ and the latter is closed by Claim 2, we conclude that $\overline{\cup_n F(\alpha_n, X)}\subset F(\mu,X)$. To show the reverse inclusion, fix $f\in F(\mu,X)$ and a standard open neighborhood $U$ of $f$. We may assume that $U=U_1\cap U_2$, where $U_1=U_1(p_1,...,p_n, B_1,...,B_n)$ and 
$U_2=U_2(q_1,...,q_k, O_1,..., O_k)$ for some $p_1,...,p_n\in X\times \mu$ and $q_1,...,q_k\in (X\times \tau ) \setminus (X\times \mu)$. To show that $U$ meets $\cup_n F(\alpha_n, X)$ let $N$ be such that $p_1,...,p_n\in X\times \alpha_N$. Define $g$ by letting $g(x,\beta) = f(x,\mu)$ for every $\beta>\alpha_N$ and $g(x,\beta)=f(x,\beta)$ for every $\beta\leq \alpha_N$. Clearly, $g\in U\cap F(\alpha_n,X)$.
\end{proof}

\par\bigskip\noindent
The following is a by-product of  the argument of Theorem \ref{thm:timesomegaone}.

\par\bigskip\noindent
\begin{thm}\label{thm:timesomegaoneplusone}
Let $X$ and $Y$ have networks of cardinality less than $\tau$, where $\tau$ is regular and uncountable. Let $C_p(X\times (\tau + 1))$ be homeomorphic to $C_p(Y\times (\tau + 1))$. Then $C_p(X\times \lambda)$ is homeomorphic to $C_p(Y\times \lambda)$ for some non-zero $\lambda<\tau$.
\end{thm}

\par\bigskip\noindent
One may wonder if a homeomorphism between $C_p(X\times \kappa)$ and $C_p(Y\times \kappa)$ for some ordinal $\kappa$ implies the existence of a homeomorphism between $C_p(X\times \lambda)$ and $C_p(Y\times \lambda)$ for some $0<\lambda<\kappa$. It is, however, not the case. Note that the function spaces over any  discrete space $X$  is homeomorphic  to $\mathbb R^{|X|}$. Therefore, $C_p(2\times \omega)$ and $C_p(3\times \omega)$ are homeomorphic but $C_p(2\times n)$ and $C_p(3\times n)$ are not homeomorphic for any positive integer $n$. 

\par\bigskip\noindent
For future reference let us isolate the case of $\tau=\omega_1$ of Theorem  \ref{thm:timesomegaone} into a corollary.

\par\bigskip\noindent
\begin{cor}\label{cor:timesomegaone}
Let $X$ and $Y$ have countable networks. If $C_p(X\times\omega_1 )$ is homeomorphic to $C_p(Y\times \omega_1 )$, then $C_p(X\times\lambda )$ is homeomorphic to $C_p(Y\times \lambda )$ for some non-zero countable  ordinal $\lambda$.
\end{cor}

\par\bigskip\noindent
The above corollary prompts the following question.
\par\bigskip\noindent
\begin{que}
Let $C_p(X\times (\omega +1))$ and $C_p(Y\times (\omega +1))$ be homeomorphic. Is it true that 
$C_p(X\times \mathbb Q)$ and $C_p(Y\times \mathbb Q)$ are homeomorphic? What additional conditions on $X$ and $Y$  guarantee an affirmative answer?
\end{que}

\par\bigskip\noindent
An argument similar to that  in Theorem \ref{thm:timesomegaone} leads to our next statement. Since the argument of Theorem  \ref{thm:timesomegaone} needs only  notation changes for our next result, we will provide the strategy without repeating the proof details of Theorem  \ref{thm:timesomegaone}.
 
\par\bigskip\noindent
\begin{thm}\label{thm:countablenetworktimesAtau}
Let $X$ and $Y$ have countable networks  and let $C_p(X\times A(\tau))$ be homeomorphic to $C_p(Y\times A(\tau))$ for some infinite cardinal $\tau$. Then $C_p(X\times A(\omega ))$ is homeomorphic to $C_p(Y\times A(\omega))$.
\end{thm}
\begin{proof}
 We may assume that $\tau$ is uncountable.
Fix a homeomorphism $\phi: C_p(X\times A(\tau))\to C_p(Y\times A(\tau))$. 
Given $A\subset \tau$, define $F(A,X)\subset C_p(X\times A(\tau ))$ as follows:
 
\par\bigskip\noindent
{\it Definition of $F(A,X)$: $f\in C_p(X\times A(\tau ))$ is in $F(A,X)$ if and only if $f(x,\beta)=f(x,\infty )$ for every $\beta\not \in A$ 
and every $x\in X$.}

\par\bigskip\noindent
{\it Claim 1. For every $f\in C_p(X\times A(\tau ))$ there exists a countable $A\subset \tau$ such that $f\in F(A,X)$.  }
\par\medskip\noindent
To prove the claim, assume the contrary. Then there exists a subset $\{\langle x_\alpha, \lambda_\alpha\rangle :\alpha<\omega_1\}\subset X\times A(\tau )$ such that $f(x_\alpha, \lambda_\alpha)\not = f(x_\alpha, \infty )$ for each $\alpha<\omega_1$, and $\lambda_\alpha\not = \lambda_\beta$ for distinct $\alpha,\beta<\omega_1$. Without loss of generality we may assume that there exist $p,q\in \mathbb Q$ such that $f(x_\alpha, \lambda_\alpha) <p <q<f(x_\alpha, \infty)$ for all $\alpha<\omega_1$. Since $X$ is Lindel\"of, there exists $\langle x,\infty\rangle $ a complete accumulation point of $\{\langle x_\alpha, \infty\rangle: \alpha<\omega_1\}$, and therefore, $f(x,\infty)\geq q$. By the product topology, $\langle x,\infty\rangle $ is also a complete accumulation point for $\{\langle x_\alpha, \lambda\rangle : \alpha<\omega_1\}$, and therefore, $f(x,\infty)\leq p$. We have arrived at a contradiction with $p<q$, which proves the claim.

\par\bigskip\noindent
The next two claims (as Claim 1) are proved using the same arguments as the corresponding claims of Theorem \ref{thm:timesomegaone}.
\par\bigskip\noindent
{\it Claim 2.  Let $K$ be an infinite subset of $\tau$ of cardinality $\kappa$. Then $F(X,K)$ is homeomorphic to $C_p(X\times A(\kappa ))$. }

\par\bigskip\noindent
{\it Claim 3. $F(X,A)$ is closed in $C_p(X\times A(\tau ))$ for any $A\subset \tau$.}

\par\bigskip\noindent
By Claim 2, it remains to find a countable $A\subset \tau$ such that $\phi(F(X,A))=F(Y,A)$. Inductively, we will define $A_n$ that will be building blocks for our desired $A$.
\begin{description}
	\item[\underline{\rm Step 0}]  $A_0=\omega$.
	\item[\underline{\rm Assumption}] $A_k$ is defined for all $k=0,..., n-1$ and $A_m\subset A_k$ if $m<k$.
	\item[\underline{\rm Step n}] Our definition of $A_n$ will depend on whether $n$ is odd or even.

\par\medskip\noindent
Assume that $n$ is odd. Since $X$ has a countable network, by Claim 2, $F(A_{n-1}, X)$ is separable. By Claims 1 and 3, there exists a countable $A_n\subset \tau$ that contains $A_{n-1}$ such that $\phi(F(A_{n-1}, X))\subset F(A_n, Y)$.

\par\medskip\noindent
Assume that $n$ is even. Similarly, we choose a countable $A_n\subset \tau$ such that $A_n$ contains $A_{n-1}$ and  $\phi^{-1}(F(A_{n-1}, Y))\subset F(A_n, X)$.
\end{description}

\par\bigskip\noindent
Our induction construction is complete. Put $A=\cup_n A_n$. By Claim 2, it remains to show that $\phi(F(A, X)) =  F(A, Y)$. The argument is identical to that in Theorem \ref{thm:timesomegaone}.
\end{proof}

\par\bigskip\noindent
A more general statement is as follows.
\begin{thm}\label{thm:timesAkappa}
Let $\tau$ be an infinite cardinal. Let $X$ and $Y$ have networks of cardinality $\tau<\kappa$. If $C_p(X\times A(\kappa ) )$ is homeomorphic to $C_p(Y\times A(\kappa ) )$, then $C_p(X\times A(\tau ) )$ is homeomorphic to $C_p(Y\times A(\tau ) )$. 
\end{thm}

\par\bigskip\noindent
It would be interesting to consider scenarios in a direction somewhat opposite to Theorems \ref{thm:timesomegaone}  - \ref{thm:timesAkappa}, namely, in a direction aligned with the following question.
\par\bigskip\noindent
\begin{que}\label{que:timesomegaone}
Assume that $C_p(X\times \alpha )$ is homeomorphic to $C_p(Y\times \alpha)$ for every $\alpha<\omega_1$. Is it true that $C_p(X\times \omega_1)$ is homeomorphic to $C_p(Y\times\omega_1)$? What if $\omega_1$ is replaced by an arbitrary fixed ordinal?
\end{que} 

\par\bigskip\noindent
Note that there are many examples of pairs $X$ and $Y$ that satisfy the assumption of Question \ref{que:timesomegaone}. For example, it was shown
by van Mill \cite{vM0} that any two non-locally compact subspaces of rational numbers have homeomorphic $C_p$'s. This result was later generalized to any infinite non-discrete  subspaces of rationals  
by Dobrowolski, Marciszewsi, and Mogilski \cite {DMM}. In connection with these nice facts and the statement of Theorem \ref{cor:timesomegaone}, the following special case of  Question \ref{que:timesomegaone} might be of interest:

\begin{que}\label{que:countabletimesomegaone}
Let $X$ and $Y$ be infinite  non-discrete subspaces of rationals.  Is it true that $C_p(X\times \omega_1)$ is homeomorphic to $C_p(Y\times\omega_1)$? What if $\omega_1$ is replaced by $\omega_1 + 1$ or any other larger ordinal?
\end{que}

\par\bigskip\noindent
Next, we would like to remark on a positive outcome in the scenario "from smaller to larger" ordinal. A particular case of Okunev's main result in \cite{O} states that if $C_p(X)$ and $C_p(Y)$ are homeomorphic for preudocompact $X$ and $Y$, then $C_p(\beta X)$ and $C_p(\beta Y)$ are homeomorphic too. It is also a classical fact that if $X$ is compact, then $X\times \omega_1$ is pseudocompact  and $\beta (X\times \omega_1) = X\times (\omega_1 + 1)$. These facts imply that if 
 $X$ and $Y$ are compact and $C_p(X\times \omega_1 )$ is homeomorphic to $C_p(Y\times \omega_1)$, then $C_p(X\times (\omega_1+1))$ is homeomorphic to $C_p(Y\times (\omega_1+1))$. Clearly, the statement holds if we replace $\omega_1$ with any ordinal of uncountable cofinality. Let us summarize this discussion as follows.

\par\bigskip\noindent
\begin{rem}\label{thm:omegaplusone} (a Corollary to Okunev's Theorem in \cite{O}) 
Let $X$ and $Y$ be compact spaces and let an ordinal $\tau$ have uncountable cofinality. If  $C_p(X\times \tau )$ is homeomorphic to $C_p(Y\times \tau)$, then $C_p(X\times (\tau+1))$ is homeomorphic to $C_p(Y\times (\tau+1))$.
\end{rem}

\par\bigskip\noindent
A few natural questions arise from the above discussion.
\par\bigskip\noindent
\begin{que}
Let $X$ and $Y$ be compact spaces. Suppose that $C_p(X\times (\omega_1+1))$ is homeomorphic to $C_p(Y\times (\omega_1+1))$. Is it true that  $C_p(X\times \omega_1 )$ is homeomorphic to $C_p(Y\times \omega_1)$?
\end{que}

\par\bigskip\noindent
\begin{que}
Are there spaces $X$ and $Y$ (not necessarily compact) such that   $C_p(X\times (\omega_1+1))$ is homeomorphic to $C_p(Y\times (\omega_1+1))$ but  $C_p(X\times \omega_1 )$ is not homeomorphic to $C_p(Y\times \omega_1)$? Or, such that  $C_p(X\times \omega_1 )$ is  homeomorphic to $C_p(Y\times \omega_1)$ but  $C_p(X\times (\omega_1+1))$ is not homeomorphic to $C_p(Y\times (\omega_1+1))$ ?
\end{que}

\par\bigskip\noindent
We would like to finish with two general questions related to our study that are also particular cases of the general problem that motivated this work.

\par\bigskip\noindent
\begin{que}
Assume that $C_p(X\times \tau)$ and $C_p(Y\times\tau)$ are homeomorphic for some ordinal $\tau>1$. What additional conditions on $X,Y,\tau$ guarantee that  $C_p(X)$ and $C_p(Y)$ homeomorphic? That $X$ and $Y$ are homeomorphic? 
\end{que}

\par\bigskip\noindent
\begin{que}\label{que:q1}
Is there an example of $X$ and $Y$ with homeomorphic $C_p$'s such that $C_p(X\times \tau)$ and $C_p(Y\times \tau)$ are not homeomorphic for some $\tau$?
\end{que}

\par\bigskip\noindent
 In connection with this question, a notable  pair to test is the set of reals and a closed segment. It is due to Gulko and Khmyleva \cite{GH} that $C_p(\mathbb R)$ and $C_p([0,1])$ are homeomorphic. It might be interesting to know  if there is an ordinal $\lambda$ such that $C_p(\mathbb R\times \lambda)$ and $C_p([0,1]\times \lambda)$  are not homeomorphic. For an overview of non-homeomorphic spaces with homeomorphic $C_p$'s , 
which may address  Question \ref{que:q1},
we refer the reader to \cite{vM}.

\par


\bigskip
\begin{thebibliography}{99}

\bibitem{Arh}
A. Arhangelskii, {\it Topological Function Spaces}, Math. Appl., vol. 78, Kluwer Academic
Publishers, Dordrecht, 1992

\bibitem{DMM}
T. Dobrowolski, W. Marcsiszewski, and J. Mogilski, {\it On topological classification of functions spaces $C_p(X)$ of low Borel complexity}, Trans. Amer. Math. Soc. 678(1991), 307-324

\bibitem{Eng}
R. Engelking, {\it General Topology}, PWN, Warszawa, 1977

\bibitem{GH}
S. P. Gulko and T. E. Khmyleva,  {\it Compactness is not preserved by t-equivalence}, Mat. Zametki, 39 (1986), 895-903, in Russian

\bibitem{O}
O. G. Okunev, {\it Spaces of functions in the topology of pointwise convergence: 
the Hewitt extension and tau-continuous functions}, 
Vestnik Moskov. Univ. Ser. I Mat. Mekh. 1985, no. 4, 78–80, in Russian


\bibitem{vM0}
J. van Mill, {\it Topological Equivalence of Certain Function Spaces}, Composito. Math., 63(1987), 195-201

\bibitem{vM}
J. van Mill, {\it The Infinite- Dimensional Topology of Function Spaces}, Elsevier, 2001
\end{thebibliography}
\end{document}